\documentclass[11pt]{article}
\usepackage{amsmath, amsthm, amssymb,latexsym,enumerate}
\usepackage{graphicx}
\usepackage{hyperref}
\usepackage{xcolor}
\usepackage{todonotes}
\usepackage{cite}

\newtheorem{theorem}{Theorem}[section]
\newtheorem{lemma}[theorem]{Lemma}
\newtheorem{corollary}[theorem]{Corollary}
\newtheorem{claim}[theorem]{Claim}
\newtheorem{conjecture}[theorem]{Conjecture}

\title{Improvements on Hippchen's Conjecture}

\author{Eun-Kyung Cho\thanks{
Department of Mathematics, Hankuk University of Foreign Studies, Yongin-si, Gyeonggi-do, Republic of Korea.
 \texttt{ekcho2020@gmail.com}
}
\and Ilkyoo Choi\thanks{
Department of Mathematics, Hankuk University of Foreign Studies, Yongin-si, Gyeonggi-do, Republic of Korea.
\texttt{ilkyoo@hufs.ac.kr}
}
\and Boram Park\thanks{
Department of Mathematics, Ajou University, Suwon-si, Gyeonggi-do, Republic of Korea.
\texttt{borampark@ajou.ac.kr}
}}
\date\today

\begin{document}

\maketitle

\begin{abstract}
Let $G$ be a $k$-connected graph on $n$ vertices. 
Hippchen's Conjecture~\cite{hippchen2008intersections} states that two longest paths in $G$ share at least $k$ vertices. 
Guti\'errez~\cite{gutierrez2020intersection} recently proved the conjecture when $k\leq 4$ or $k\geq \frac{n-2}{3}$. 
We improve upon both results; namely, we show that two longest paths in $G$ share at least $k$ vertices when $k=5$ or $k\geq \frac{n+2}{5}$. 
This completely resolves two conjectures in~\cite{gutierrez2020intersection} in the affirmative.
\end{abstract}

\section{Introduction}\label{section-introduction}

It is easy to see that two longest paths in a connected graph must have a common vertex. 
In 1966, Gallai~\cite{MR0232693} asked if all longest paths in a connected graph must have a common vertex. 
This is known to be true for connected series-parallel graphs~\cite{2017ChEhFeHeShYaYa}, yet, it is not true in general. 
The minimum order of a graph answering Gallai's question in the negative has 12 vertices, found by Walther and Voss~\cite{walther1974kreise} and independently by~Zamfirescu~\cite{1976Zamfirescu}. 
There is a rich literature regarding the history of intersections of longest paths.
We direct the readers to the surveys~\cite{2001Zamfirescu,2013ShZaZa} for more details and related results. 

Instead of ensuring a common vertex in all (or many) longest paths, one could seek sufficient conditions that imply many common vertices in two longest paths. 
Hippchen~\cite{hippchen2008intersections} made the following intriguing conjecture, suggesting that the number of vertices shared by two longest paths in a graph $G$ is  at least the connectivity of $G$.
Note that the complete bipartite graph where the two parts have $k$ and $2k+2$ vertices demonstrates that the conjecture, if true, is tight. 
There is also an infinite family of graphs for each $k$ that demonstrates the tightness of the conjecture~\cite{gutierrez2020intersection}.

\begin{conjecture}[\cite{hippchen2008intersections}]\label{conj}
If $G$ is a $k$-connected graph, then two longest paths in $G$ share at least $k$ vertices.
\end{conjecture}

Hippchen~\cite{hippchen2008intersections} proved the above conjecture for  $k \le 3$, and recently, Guti\'{e}rrez~\cite{gutierrez2020intersection} proved the conjecture for $k = 4$.

\begin{theorem}[\cite{hippchen2008intersections,gutierrez2020intersection}]\label{thm:k=4}
For $k \le 4$, Conjecture~\ref{conj} is true.
\end{theorem}
Our first result verifies Hippchen's Conjecture for the next open case, which was a conjecture in~\cite{gutierrez2020intersection}.

\begin{theorem}\label{thm:k5}
For $k=5$, Conjecture~\ref{conj} is true. 
\end{theorem}

Guti\'errez also proved that the same conclusion holds whenever the connectivity is linear in terms of the number of vertices. 
Namely, for a $k$-connected graph $G$ on at most $3k+2$ vertices, two longest paths in $G$ share at least $k$ vertices.
He actually proved the following stronger statement: If $G$ is a $k$-connected graph on $n$ vertices, then two longest paths in $G$ share at least $\left\lceil\frac{8k-n+2}{5}\right\rceil$ vertices. 
We improve upon the aforementioned statement in the following form: 

\begin{theorem}\label{thm:main}
If $G$ is a $k$-connected graph on $n$ vertices, then two longest paths in $G$ share at least $\min\left\{ k, \left\lceil\frac{8k-n-4}{3}\right\rceil\right\}$ vertices. 
\end{theorem}

Since $k\ge \frac{n+2}{5}$ implies $\min\left\{ k, \left\lceil\frac{8k-n-4}{3}\right\rceil\right\} =k$, we obtain the following corollary as a direct consequence of Theorem~\ref{thm:main}, thereby extending the class of graphs for which Hippchen's Conjecture is true.
This resolves a proposed conjecture in~\cite{gutierrez2020intersection} in a stronger form. 

\begin{corollary}\label{cor:main}
For a $k$-connected graph on at most $5k-2$ vertices, Conjecture~\ref{conj} is true.
\end{corollary}

In Section~\ref{sec:pre}, we give notation and definitions used throughout the paper, and lemmas that are useful in our proof. 
In Section~\ref{sec:general}, we prove Theorem~\ref{thm:main}, and 
the proof of Theorem~\ref{thm:k5} is presented in Section~\ref{sec:k=5:by Boram}.
We conclude the paper with some remarks in Section~\ref{sec:remarks}.

\section{Preliminaries}\label{sec:pre}

We consider only finite simple graphs.
For a graph $G$, let $V(G)$ and $E(G)$ denote its vertex set and edge set, respectively.
Given $S \subseteq V(G)$, we use $G-S$ to denote the subgraph of $G$ induced by $V(G)\setminus S$.
Given a path $P$, the {\it length} of $P$, denoted $\ell(P)$, is the number of edges of $P$. 
For two vertices $x$ and $y$ on a path $P$, let $P_{[x,y]}$ denote the subpath of $P$ from $x$ to $y$. 
We use $P_{(x,y)}$ to denote $P_{[x,y]}-\{x,y\}$.
Given $S, T \subseteq V(G)$, an {\it $(S,T)$-path} is a path from a vertex in $S$ to a vertex in $T$ where none of its internal vertices are in $S\cup T$.
For brevity, if $S =\{x\}$ and $T = \{y\}$, then we say a {\it $(x,y)$-path}.
If $P$ and $Q$ are two paths starting from $x$ such that the only common vertex is $x$, then let $P+Q$ denote the path whose vertex set is $V(P) \cup V(Q)$ and edge set is $E(P) \cup E(Q)$.
Also, for a path $P$ and its subpath $R$, we use $P-R$ to denote the subgraph of $P$ induced by $V(P) \setminus V(R)$.

We end this section with a useful lemma that will be used in the forthcoming proofs.

\begin{lemma}\label{claim-structure}
Let $G$ be a $k$-connected graph, where $P$ and $Q$ are longest paths in $G$. 
If $|V(P)\cap V(Q)|\leq k-1$, then there exists a longest path $Q':q'_0q'_1\cdots q'_\ell$ in $G$ such that $V(Q)=V(Q')$,
$q'_0,q'_i\not\in V(P)$, $q'_i$ is a neighbor of $q'_0$, and either $i=1$ or $q'_0q'_{i+1}\in E(G)$ for some $i\in\{2,\ldots,\ell-1\}$. 
\end{lemma}
\begin{proof} 
Suppose that $|V(P)\cap V(Q)|\le k-1$. 
 Let $Q:q_0q_1\cdots q_\ell$.
 Since $G$ is $k$-connected, $q_0$ has at least $k$ neighbors, all of which must be on $Q$ because $q_0$ is an end of $Q$, and $Q$ is a longest path.
Also, $q_0$ must have a neighbor $q_{i+1}$ ($i\ge 0$) such that $q_{i} \in V(Q)\setminus V(P)$ since $|V(P)\cap V(Q)|\le k-1$.
Now, the path $Q': Q_{[q_i,q_{0}]}+q_0q_{i+1}+Q_{[q_{i+1},q_\ell]}$ is a longest path in $G$ since $V(Q')=V(Q)$. 
Note that $q_i$, which is the first vertex of $Q'$, is not on $P$. 
Applying the same process to $q_l$, which is the other end of $Q'$,
we can find a longest path in $G$ whose vertex set is $V(Q)$ where both end vertices are not on $P$; let $Q':q'_0q'_1\cdots q'_{\ell}$ be such a path. 
That is, $Q'$ is a longest path in $G$ such that $V(Q)=V(Q')$ and $q_0',q'_\ell\not\in V(P)$.
The following holds:

\begin{enumerate}[(1)]
\item  If $q'_1\not\in V(P)$ or $q'_{\ell-1}\not\in V(P)$, then $Q'$ is the path we seek, so $q'_1,q'_{\ell-1} \in V(P)$.

\item  If $q'_0q'_\ell$ is an edge, then $q'_{\ell}q'_0q'_1\cdots q'_{\ell-1}$ is the path we seek, so $q'_{\ell}\not\in N_G(q'_0)$.

\item If $q'_0$ has a neighbor $q'_j$ $(j\ge 2)$ such that $q'_{j-1}, q'_{j-2}\in V(Q)\setminus V(P)$, then $Q'_{[q'_{j-1},q'_0]} +q'_0q'_j+Q'_{[q'_j,q'_\ell]}$ is the path we seek, so for every neighbor $q'_j$ of $q'_0$ with $j\ge 2$, 
either $q'_{j-1} \in V(P)$ or $q'_{j-2} \in V(P)$.
\end{enumerate}

In other words, since $q'_{\ell}\not\in N_G(q'_0)$, we know
$N_G(q'_0)\setminus\{q'_1\}$ is contained in $X=\{ q'_{i},q'_{i+1} \mid q'_{i-1}\in V(P)\cap V(Q), i\neq \ell \}$. 
For each $q'_{i-1}\in V(P)\cap V(Q)$ with $i\neq \ell$, 
 define\[ X_{i-1}=\begin{cases}
\{q'_i\}&\text{ if }q'_{i}\in V(P)\cap V(Q),\\
\{q'_i,q'_{i+1}\}&\text{ otherwise.}\end{cases}\]
Note that the union of the sets $X_{i-1}$'s is $X$ and the pairwise intersection of $X_{i-1}$'s is empty, so
$\{ X_{i-1}\mid q'_{i-1}\in V(P)\cap V(Q), i\neq \ell  \}$ is a partition of $X$ with at most $k-2$ parts.
Since $|N_G(q'_0)\setminus\{q'_1\}|\ge k-1$, by the Pigeonhole Principle,
there is some $q'_{i-1}\in V(P)\cap V(Q)$, with 
$i \neq \ell$ such that both vertices in $X_{i-1}=\{q'_{i},q'_{i+1}\}$ are neighbors of $q'_0$.
This also implies that  $q'_{i}\not\in V(P)$, so $Q'$ is the path we seek.
\end{proof}

The following well-known lemma for $k$-connected graphs will also be used. 

\begin{lemma} [The Fan Lemma \cite{bondy2008graph}]\label{fan}
Let $G$ be a $k$-connected graph.
If $x$ is a vertex of $G$, and $Y \subseteq V(G) \setminus \{x\}$ is a set of at least $k$ vertices of $G$, then there exists a family of $k$ internally disjoint $(x,Y)$-paths whose terminal vertices are distinct.
\end{lemma}

\section{Proof of Theorem~\ref{thm:main}}\label{sec:general}

In this section, we prove Theorem~\ref{thm:main}. The following lemma contains a key idea to prove Theorem~\ref{thm:main}.

\begin{lemma}\label{lem:submain}
Let $G$ be a $k$-connected graph with $k\ge 3$. 
If $P$ and $Q$ are two longest paths in $G$, then $P$ and $Q$ share at least $\min\{4k-\ell-3,k\}$ vertices, where $\ell=\ell(P)=\ell(Q)$.
\end{lemma}

\begin{proof} 
Let $P$ and $Q$ be longest paths in $G$ that share at most $k-1$ vertices. Note that it is enough to show that $|V(P)\cap V(Q')|\ge 4k-\ell-3$ for some longest path $Q'$ with $V(Q)=V(Q')$.
By Lemma~\ref{claim-structure}, we may assume that  $Q:q_0q_1\cdots q_\ell$ is a longest path such that 
$q_0,q_i\not\in V(P)$, $q_i$ is a neighbor of $q_0$, and either $i=1$ or $q_0q_{i+1}\in E(G)$ for some $i\in\{2,\ldots,\ell-1\}$. 
Let $P$ be an $(x, y)$-path. 

\begin{claim}\label{claim2}
For an edge $qq'\in E(G)$ such that $q,q'\in V(Q)\setminus V(P)$, it holds that
\[ \ell\ge 4k-3-  |N_G(\{ q,q'\}) \cap V(P)|.\]
\end{claim}

\begin{proof}
By the Fan Lemma (Lemma~\ref{fan}), there exists a family of $k$ internally disjoint $(q,V(P))$-paths whose terminal vertices on $P$ are distinct.
Among these $k$ paths, by ignoring a path containing $q'$, we obtain a family of $k-1$ internally disjoint $(q,V(P))$-paths $R_1, \ldots, R_{k-1}$ in $G-\{q'\}$ where each $R_i$ is a $(q, v_i)$-path, and $v_1, \ldots, v_{k-1}$ are distinct.
By the same argument, there exists a family of $k-1$ internally disjoint $(q',V(P))$-paths $R'_1, \ldots, R'_{k-1}$ in $G-\{q\}$ where each $R'_i$ is a $(q', v'_i)$-path, and $v'_1, \ldots, v'_{k-1}$ are distinct. 
See Figure~\ref{fig:(q,P)-path}.
\begin{figure}[h!]
  \centering
  \includegraphics[width=11cm, page = 1]{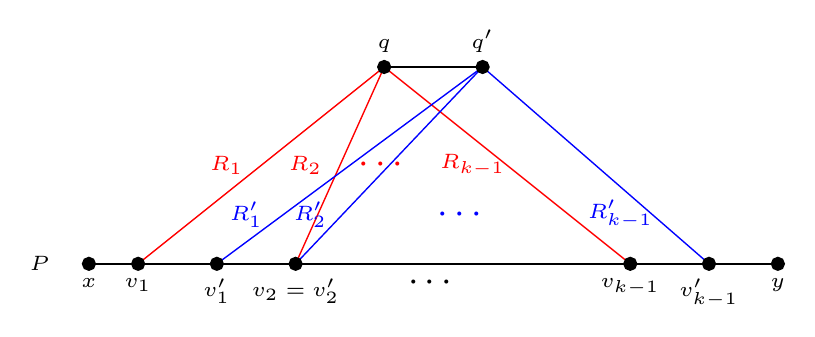}
  \caption{An illustration for $R_1, \ldots, R_{k-1}$ and $R'_1, \ldots, R'_{k-1}$}
  \label{fig:(q,P)-path}
\end{figure}

We may assume that $x,v_1,\ldots,v_{k-1},y$ lie on $P$ in this order, and $x,v'_1,\ldots,v'_{k-1},y$ also lie on $P$ in this order.

Let $I=\{ v_i\mid \ell(R_i)=1\}$ and $I'=\{ v'_i\mid \ell(R'_i)=1\}$. 
Since $P$ is a longest path, the following holds:
\begin{enumerate}[(1)]
\item Since $q'q+R_1+P_{[v_1,y]}$ is a path in $G$,  $\ell(P_{[x,v_1]}) \ge \ell(R_1)+1$. 
Similarly, since $q'q+R_{k-1}+P_{[v_{k-1},x]}$ is a path in $G$,  $\ell(P_{[v_{k-1},y]}) \ge \ell(R_{k-1})+1$.
\item If $P':P_{[x,u]}+W+P_{[w,y]}$ is a path in $G$ for some $(u,w)$-path $W$, then $\ell(P_{[u,w]})\ge \ell(W)$.
Since $R_{i}+R_{i+1}$ is a $(v_i,v_{i+1})$-path in $G$ that is internally disjoint with $P$, we obtain $\ell(P_{[v_{i},v_{i+1}]}) \ge \ell(R_{i})+\ell(R_{i+1})$.
If $v_i\in I \cap I'$, then $v_i=v'_j$ for some $j$, so $R'_{j}+q'q+R_{i+1}$ is a $(v_i,v_{i+1})$-path in $G$ that is internally disjoint with $P$, so $\ell(P_{[v_{i},v_{i+1}]}) \ge \ell(R'_{j})+\ell(R_{i+1})+1=\ell(R_{i})+\ell(R_{i+1})+1$ since $\ell(R'_j)=\ell(R_i)=1$. 
Similarly, $\ell(P_{[v_{i},v_{i-1}]}) \ge \ell(R_{i})+\ell(R_{i-1})+1$.
\end{enumerate}
Thus, we have the following:
\begin{eqnarray*}
\ell & = & \ell(P_{[x,v_1]}) + \sum_{i=1}^{k-2} \ell(P_{[v_i,v_{i+1}]}) + \ell(P_{[v_{k-1},y]})   \\
&\ge & (\ell(R_1)+1) + \sum_{i=1}^{k-2} (\ell(R_i)+\ell(R_{i+1}) ) 
 +(\ell(R_{k-1})+1)+ (|I\cap I'|-1)\\
&= & 2 \sum_{i=1}^{k-1}\ell(R_i) + |I\cap I'|+1 \\
&\ge & 2|I|+4(k-1-|I|) +   |I\cap I'|+1,  
\end{eqnarray*}
where the first inequality is from (1) and (2), and the last inequality is from the fact that  $\ell(R_i)=1$  if and only if $v_i\in I$.
By applying the same argument to $q'$, we obtain
\[ \ell \ge  2|I'|+4(k-1-|I'|) +  |I\cap I'|+1. \]
Hence, by adding the two inequalities and dividing by $2$ to each side, we have 
\[  \ell \ge -|I|-|I'|+4(k-1) + |I\cap I'|+1 = 4k-3-|I\cup I'|.\]
Since a vertex in $I\cup I'$ is a neighbor of $q$ or $q'$, we obtain
$|I\cup I'|\le |N_G(\{ q,q'\}) \cap V(P)|$, which implies the desired conclusion.
\end{proof}

By Claim~\ref{claim2}, it is sufficient to show that  \begin{eqnarray}\label{eq:last}
&&|N_G(\{ q_0,q_i\}) \cap V(P)| \le |V(Q)\cap V(P)|.
\end{eqnarray}
Note that all neighbors of an end vertex of a longest path must lie on the path. 
Thus, $N_G(q_0)\subset V(Q)$ since $q_0$ is an end vertex of $Q$. 
If $q_0q_{i+1}\in E(G)$, then  $Q_{[q_i,q_0]}+q_0q_{i+1}+Q_{[q_{i+1},q_{\ell}]}$ is also a longest path in $G$ whose vertex set is $V(Q)$, so $N_G(q_i)\subset V(Q)$. 
Hence, in this case, $N_G(\{ q_0,q_i\}) \cap V(P)$ is a subset of  $V(Q)\cap V(P)$, so \eqref{eq:last} holds.

Now suppose that $i=1$. 
For simplicity, let 
\begin{eqnarray*}
N_0&:=&N_G(\{q_0,q_1\})\cap V(P)\cap V(Q),\\
N_1&:=& (N_G(\{q_0,q_1\})\cap V(P) )\setminus V(Q).\end{eqnarray*}
Since $\{N_0,N_1\}$ is a partition of $N_G(\{q_0,q_1\})\cap V(P)$,~\eqref{eq:last} is equivalent to the following:
\[ |N_0|+ |N_1| \le |V(P)\cap V(Q)|.\]
 For every vertex $p\in V(P)$, let $p^+$ denote the neighbor of $p$ on $P$ such that $\ell(P_{[x,p]})+1 = \ell(P_{[x,p^+]})$.

Since $q_0, q_1 \notin V(P)$ by the maximality of $P$, we know $\{x,y\}$ is disjoint with $N_1$, which implies that $p^+$ is well-defined for every $p\in N_1$.
If $p^+\not\in V(Q)$ for some $p\in N_1$, then 
$p^+pq_1+Q_{[q_1,q_\ell]}$ is a path longer than $Q$, which is a contradiction.  
Thus $p^+$ is on $Q$ for every $p\in N_1$.
Hence, we have a well-defined function $f: N_0\cup N_1 \rightarrow V(P)\cap V(Q)$: 
\[ f(p) =\begin{cases}
p &\text{ if }p\in N_0,\\
p^+ &\text{ if }p \in N_1.\end{cases}\]
We complete the proof by showing that $f$ is injective. 
By the definition of $f$, it is sufficient to show that $f(N_0)$ and $f(N_1)$ are disjoint.
Suppose that $a\in f(N_0)\cap f(N_1)$, which implies that 
 $a\in N_0$ and $a=p^+$ for some $p\in N_1$.
Since $N_G(q_0)\subset V(Q)$,
$N_1=(N_G(\{q_1\})\cap V(P) )\setminus V(Q)$, and it follows that  $pq_1\in E(G)$.
If $a\in N_G(q_0)$, then replacing the edge $pp^+$ from $P$ with $pq_1q_0p^+$   gives a path longer than $P$, which is a contradiction.
If $a\in N_G(q_1)$, then replacing the edge $pp^+$ from $P$ with $pq_1p^+$ gives a path longer than $P$, which is a contradiction.
Hence,  $f$ is injective and therefore  $|N_0|+|N_1|\le |V(P)\cap V(Q)|$, and hence \eqref{eq:last} holds.
\end{proof}

From Lemma~\ref{lem:submain}, we prove Theorem~\ref{thm:main}.

\begin{proof}[Proof of Theorem~\ref{thm:main}]
Let $\ell=\ell(P)=\ell(Q)$.
By Lemma~\ref{lem:submain}, we have $|V(P)\cap V(Q)|\ge \min\{4k-\ell-3,k\}$. 
If $\ell \le 3k-3$, then $\min\{4k-\ell-3,k\}\ge k$, so suppose that $\ell \ge 3k-2$.

Now, 
$|V(P)\cap V(Q)| = |V(P)|+|V(Q)|-|V(P)\cup V(Q)|\ge 2(\ell+1)-n$.
For simplicity, let $\alpha=\min\{4k-\ell-3,k\}$ and $\beta=2\ell+2-n$. 
Together with Lemma~\ref{lem:submain}, $|V(P)\cap V(Q)|\ge \max\{\alpha,\beta\}$  and so 
\[ |V(P)\cap V(Q)|\ge \frac{ 2\alpha+\beta}{3} \ge 
 \min\left\{  \frac{2(4k-\ell-3)+(2\ell+2-n)}{3} , \frac{2k+2\ell+2-n}{3} \right\}.
 \]
Since $\ell\ge 3k-2$, it holds that $\frac{2(4k-\ell-3)+(2\ell+2-n)}{3}=\frac{8k-n-4}{3} = \frac{ 2k+2(3k-3)+2-n}{3} < \frac{2k+2\ell
+2-n}{3}$.
Hence, $ |V(P)\cap V(Q)|\ge  \frac{ 2\alpha+\beta}{3}\ge  \frac{8k-n-4}{3}$.
\end{proof}

\section{Proof of Theorem~\ref{thm:k5}} \label{sec:k=5:by Boram}

In this section, we prove Theorem~\ref{thm:k5}, which resolves Conjecture~\ref{conj} for $k = 5$ in the 
affirmative.

\begin{proof}[Proof of Theorem~\ref{thm:k5}]
Suppose to the contrary that there is a $5$-connected graph $G$ where $P$ and $Q$ are two longest paths in $G$ that share at most four vertices. 
Let $S=V(P)\cap V(Q)$.
By Theorem~\ref{thm:k=4}, we know $|S|=4$, so let $S=\{v_1,v_2,v_3,v_4\}$.
By the argument in the first paragraph of the proof of Lemma~\ref{claim-structure}, we may assume that all end vertices of $P$ or $Q$ are not in $S$.
Let $p_0$ and $p_1$ (resp. $q_0$ and $q_1$) be the ends of $P$ (resp. $Q$).
All eight vertices $p_0$, $p_1$, $q_0$, $q_1$, $v_1$, $v_2$, $v_3$, $v_4$ must be distinct.
Assume $p_0,v_1,v_2,v_3,v_4,p_1$ lie on $P$ in this order.
Then, there is a permutation $\sigma$ on $\{1,2,3,4\}$ such that $q_0,v_{\sigma(1)},v_{\sigma(2)},v_{\sigma(3)},v_{\sigma(4)},q_1$ lie on $Q$ in this order.

By the structural symmetry of $P$ and $Q$, $\sigma, \sigma^{-1}$, and $(14)(23)\sigma$ all represent the same situation. 
Therefore, we may assume that $\sigma$ is one of the following seven permutations\footnote{The following shows the equivalence: 
\begin{center}
$\sigma_1 \sim \{(1), (14)(23)\}, \hfill
\sigma_2 \sim \{(14), (23)\},  \hfill
\sigma_3 \sim \{(13), (24), (1234), (1432)\}, \hfill
\sigma_4\sim \{(12), (34), (1324), (1423)\}, 
\newline 
\sigma_5\sim  \{(123), (124), (132), (134), (142), (143), (234), (243)\},  \hfill
\sigma_6\sim  \{(12)(34), (13)(24)\}, \hfill
\sigma_7\sim  \{(1243), (1342)\}. $
\end{center}
Here, for $\sigma \sim A$, $\sigma' \in A$ if and only if 
$\sigma' = \sigma$ or $\sigma'$ can be obtained from $\sigma$ by applying a sequence of operations consisting of inverses or composing $(14)(23)$ on the left.} (See Figure~\ref{fig:sigma}.):
\[  \sigma_1=(1),  \quad \sigma_2=(23), \quad \sigma_3=(1234), \quad \sigma_4=(12), \quad \sigma_5=(134),\quad \sigma_6=(12)(34),\quad \sigma_7=(1243). \]

\begin{figure}[h!]
  \centering
  \includegraphics[width=18cm,page=2]{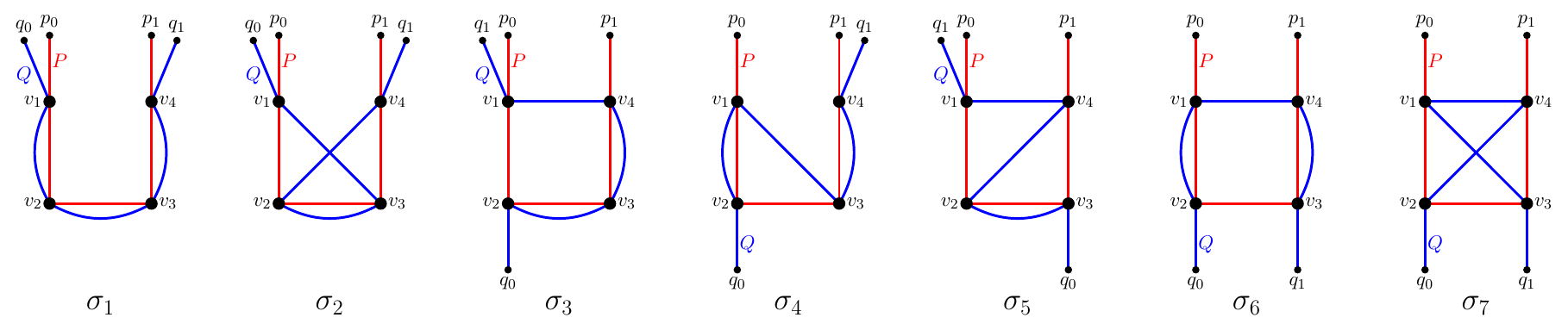}
  \caption{Illustrations for all possible seven permutations}
  \label{fig:sigma}
\end{figure}

Let $\mathcal{P}$ and $\mathcal Q$ be  the set of all components in $P-S$ and $Q-S$, respectively. 
 Define a graph $\mathcal{H}$ such that $V(\mathcal{H})=\mathcal{P}\cup \mathcal{Q}$ and $XY \in E(\mathcal{H})$ if and only if there is an $(X, Y)$-path in $G - S$  whose internal vertices are not in $V(P)\cup V(Q)$. 
For brevity,  let $P_0=P_{[p_0,v_1]}-\{v_1\}$, $P_1=P_{[v_4,p_1]}-\{v_4\}$, $Q_0=Q_{[q_0,v_{\sigma{(1)}}]}-\{v_{\sigma{(1)}}\}$, and $Q_1=Q_{[ v_{\sigma{(4)}},q_1]}-\{v_{\sigma{(4)}}\}$.
Since $p_0,p_1,q_0,q_1$ are not in $S$, each of the four paths $P_0,P_1,Q_0,Q_1$ has a vertex. 
That is, $4\le |V(\mathcal{H})|\le 10$.
Since $G$ is 5-connected, $G-S$ is connected, so $\mathcal{H}$ is connected.

Let $L$ be the subgraph of $G$ induced by the edges in $E(P)\cup E(Q)$. 
For paths $X \in \mathcal{P}$ and $Y \in \mathcal{Q}$, we say that $\{X,Y\}$ is a  \textit{replaceable} pair if one of the following holds:
\begin{enumerate}[(1)]
    \item 
    $L[V(X)\cup V(Y)\cup\{v_i,v_{i+1}\}]$ is a cycle 
    for some $i\in\{1,2,3\}$; 
    \item $\sigma(1)=1$ and $\{X,Y\}=\{P_0,Q_0\}$;   
    \item $\sigma(4)=4$ and $\{X,Y\}=\{P_1,Q_1\}$.
\end{enumerate}
For example, if $\sigma=(23)$, then  $\{P_0,Q_0\}$ and $\{P_1,Q_1\}$ are replaceable pairs.
In addition, $\{P_{(v_2,v_3)}, Q_{(v_3,v_2)}\}$ is a  replaceable pair as long as each of $P_{(v_2,v_3)}$ and $Q_{(v_3,v_2)}$ has a vertex.
In other words, if $\{X, Y\}$ is a replaceable pair, then replacing $X$ with $Y$ in $P$ becomes another longest path. 
Moreover, depending on the structural symmetry of $P$ and $Q$, there are many other longest paths that are not necessary expressed in terms of replaceable pairs. 

\begin{claim}\label{claim:XY}
For two paths $X\in \mathcal{P}$ and $Y\in \mathcal{Q}$, the following holds:
\begin{itemize}
\item[\rm (i)] 
If either $X, Y \in \{P_0,P_1, Q_0,Q_1\}$ or $L[V(X)\cup V(Y)\cup \{v_i\}]$ is a path for some $i$, then $XY\not\in E(\mathcal{H})$.
Furthermore, if $\{A,B\}$ is a replaceable pair disjoint with $\{X,Y\}$, then 
in the graph $\mathcal{H}$, each of $A$ and $B$ is adjacent to at most one vertex in $\{X,Y\}$.
\item[\rm (ii)] If $\sigma\neq \sigma_6$, then $P_0P_1\notin E(\mathcal{H})$.
\end{itemize}
\end{claim}
\begin{proof}
To show (i), we use proof by contradiction, constructing a path longer than $P$ (and $Q$).

If $XY\in E(\mathcal{H})$, where either $X, Y \in \{P_0, P_1, Q_0, Q_1\}$ or $L[V(X)\cup V(Y)\cup \{v_i\}]$ is a path for some $i$, then there is an $(X,Y)$-path $R$  whose internal vertices are not in $ V(P)\cup V(Q) $.
Let $r\in X$ and $r'\in Y$ be the ends of $R$. 

If $\{A,B\}$ is a replaceable pair disjoint with $\{X,Y\}$, such that $XA, AY \in E(\mathcal{H})$ (without loss of generality, let $A \in \mathcal{P}$),
then there is a $(X,A)$-path $R$ and a $(A,Y)$-path $W$ whose internal vertices are not in $V(P)\cup V(Q)$, where $r\in V(X)$ and $r'\in V(A)$ are ends of $R$ and $w'\in V(A)$ and $w\in V(Y)$ are ends of $W$. 
See Figure~\ref{fig:parallel}. 
\begin{figure}[h!]
  \centering
  \includegraphics[width=6.5cm, page = 3]{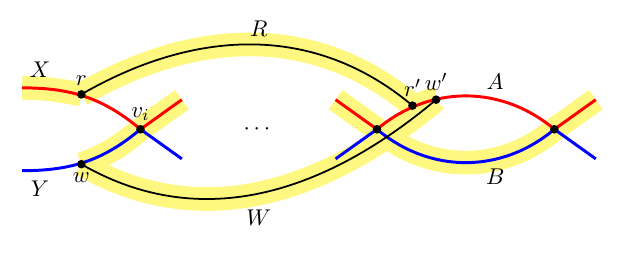}
  \caption{A replaceable pair $\{A,B\}$ such that $XA, AY \in E(\mathcal{H})$. 
  The yellow part is a path $P'$, where $L[V(X)\cup V(Y)\cup \{v_i\}]$ is a path for some $i$.}
  \label{fig:parallel}
\end{figure}
Let $R^*=R+P_{[r',w']}+W$ and $P^* = P-A+B$.
Then $P^*$ is a longest path, and $R^*$ is an $(X,Y)$-path whose internal vertices are not in $V(P^*) \cup V(Q)$, where $r \in V(X)$ and $w \in V(Y)$ are ends of $R^*$.
So, in this case, we regard $P^*$ (resp. $R^*$) as $P$ (resp. $R$).

If $X, Y \in \{P_0, P_1, Q_0, Q_1\}$
(without loss of generality, let $X = P_0$ and $Y = Q_0$), then let
\[P'=P_{[p_0,r]}+R+Q_{[r',q_1]} \qquad \text{ and } \qquad
Q'=Q_{[q_0,r']}+R+P_{[r,p_1]}. \]

If $L[V(X)\cup V(Y)\cup \{v_i\}]$ is a path for some $v_i \in S$, then let
\[P' = P-P_{[r,v_i]}+R+Q_{[r',v_i]} \qquad \text{ and } \qquad Q' = Q-Q_{[r',v_i]}+R+P_{[r,v_i]}.\]

In each case, $P'$ and $Q'$ are paths such that $\ell(P')+\ell(Q')=\ell(P)+\ell(Q)+2\ell(R)$, which leads to a contradiction, since $\max\{\ell(P'),\ell(Q')\}>\ell(P)$. 

To show (ii), if $\sigma = \sigma_i$ for some $i\in \{1,2,3,4,5\}$, then $\{P_j,Q_k\}$ is a replaceable pair for some $j, k \in \{0,1\}$ and applying (i) asserts that $P_0P_1\not\in E(\mathcal{H})$.
If $\sigma = \sigma_7$, then 
as we observed from the structural symmetry of $P$ and $Q$, we can construct other longest paths
$P' = P_{[p_0,v_1]}+Q_{[v_1,v_4]}+P_{[v_4,v_2]}+Q_{[v_2,q_0]}$ and $Q' = P_{[p_1,v_4]} + Q_{[v_4,v_2]}+P_{[v_2,v_1]}+Q_{[v_1,q_1]}$ such that $P'_0 = P_0$ and $Q'_0 = P_1$.
Thus, by (i), $P_0P_1 \notin E(\mathcal{H})$.
Hence, if it is the case where $\sigma\neq \sigma_6$, then $P_0P_1\not\in E(\mathcal{H})$.
\end{proof}

\begin{claim}\label{claim:p0}
Suppose that $\sigma=\sigma_i$ for some $i\in\{3,4,5,6,7\}$. In the graph $\mathcal{H}$, for the neighbors of $P_0$, the following holds: 
\begin{itemize}
\item[\rm (i)] If $\sigma=\sigma_i$ for some $i\in\{3,4,5,6,7\}$, then $P_0$ is not adjacent to $P_{(v_2,v_3)}$ in $\mathcal{H}$. 

\item[\rm (ii)] If $\sigma=\sigma_i$ for some $i\in\{5,6\}$, then $P_0$ is not adjacent to  $P_{(v_3,v_4)}$ in $\mathcal{H}$.  
\end{itemize}
\end{claim} 
\begin{proof}  Suppose that $P_0$ is adjacent to  $Y$ in $\mathcal{H}$ for some $Y\in \{P_{(v_2,v_3)},P_{(v_3,v_4)}\}$. Then there is an $(P_0, Y)$-path $R$ whose internal vertices are not in $V(P)\cup V(Q)$, where $r\in V(P_0)$ and $r'\in V(Y)$ are ends of $R$. 
In each case of the following, we construct paths $P'$ and $Q'$ such that $\ell(P')+\ell(Q')=\ell(P)+\ell(Q)+2\ell(R)$, which leads to a contradiction, since $\max\{\ell(P'),\ell(Q')\}>\ell(P)$. See Figures~\ref{fig:p23} and~\ref{fig:p34}, where $P'$ is a yellow path in each figure.

\begin{figure}[h!]
  \centering
  \includegraphics[width=16cm, page = 4]{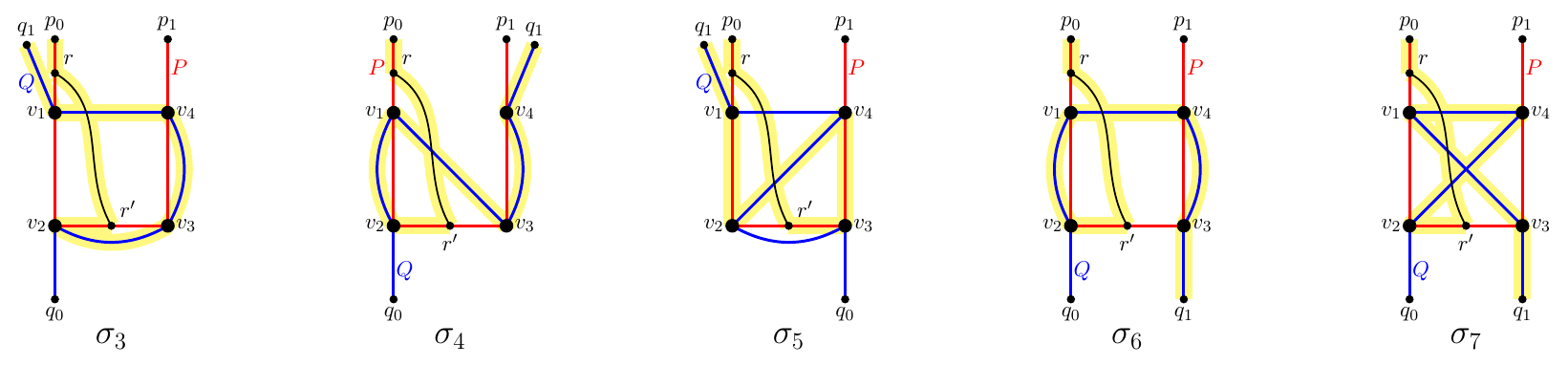}
  \caption{Illustrations for the proof of  Claim~\ref{claim:p0} (i)}
  \label{fig:p23}
\end{figure}

To check (i), let $Y=P_{(v_2,v_3)}$, $i\in\{3,4,5,6,7\}$, and see Figure~\ref{fig:p23}.
If $\sigma=\sigma_i$, where $i \in \{3,4,6,7\}$, then let
\[ P'=P_{[p_0,r]}+R+P_{[r',v_2]}+Q_{[v_2,q_1]}, \qquad Q'=Q_{[q_0,v_2]}+P_{[v_2,r]}+R+P_{[r',p_1]}.\]
If  $\sigma=\sigma_5$, then let 
\[P'=P_{[p_0,r]}+R+P_{[r',v_4]}+Q_{[v_4,v_2]}+P_{[v_2,v_1]} + Q_{[v_1,q_1]},\] \[Q'=Q_{[q_0,v_2]}+P_{[v_2,r']}+R+P_{[r,v_1]}+Q_{[v_1,v_4]}+P_{[v_4,p_1]}.\]

To check (ii), let $Y=P_{(v_3,v_4)}$, $i\in\{5,6\}$, and see Figure~\ref{fig:p34}. 
If $\sigma=\sigma_5$, then let
\[ P'=P_{[p_0,r]}+R+P_{[r',v_2]}+Q_{[v_2,q_1]}, \qquad Q'=Q_{[q_0,v_2]}+P_{[v_2,r]}+R+P_{[r',p_1]}.\]
If $\sigma=\sigma_6$, then let
 \[ P'=P_{[p_0,r]}+R+P_{[r',v_1]}+Q_{[v_1,v_4]}+P_{[v_4,p_1]}, \qquad Q'=Q_{[q_0,v_1]}+P_{[v_1,r]}+R+P_{[r',v_4]}+Q_{[v_4,q_1]}.\]
This completes the proof of the claim.
\begin{figure}[h!]
  \centering
  \includegraphics[width=7cm, page = 5]{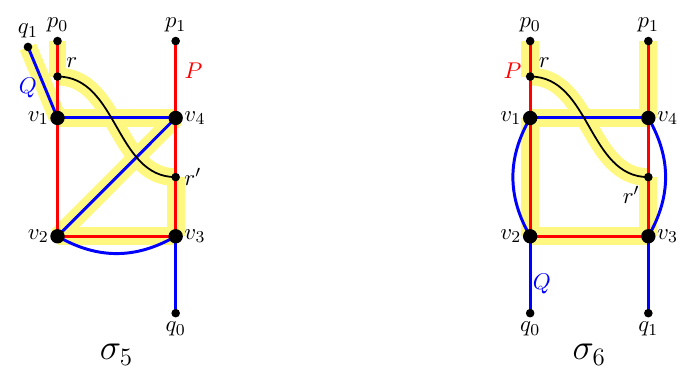}
  \caption{Illustrations for the proof of  Claim~\ref{claim:p0}~(ii)}
  \label{fig:p34}
\end{figure}
\end{proof}

We will reach a contradiction by showing that  $\sigma=\sigma_i$ is impossible for each $i\in \{1,\ldots,7\}$.  
First, suppose that $\sigma=\sigma_i$ for some $i\in \{3,4,5,6,7\}$. 
We will list all possible elements in $N_{\mathcal{H}}(P_0)$.  
By Claim~\ref{claim:XY} and the structural symmetry of $P$ and $Q$, we obtain the following:
\[ N_{\mathcal{H}}(P_0) \subseteq 
\begin{cases} 
\{P_{(v_2,v_3)}, Q_{(v_2,v_3)}, P_{(v_3,v_4)}, Q_{(v_3,v_4)}\} & \text{ if }\sigma=\sigma_3 
,\\
\{P_{(v_2,v_3)}, P_{(v_3,v_4)}, Q_{(v_3,v_4)}\} &\text{ if }\sigma=\sigma_4
, \\
\{P_{(v_2,v_3)}, Q_{(v_2,v_3)}, Q_{(v_2,v_4)}, P_{(v_3,v_4)}\} &\text{ if }\sigma=\sigma_5
, \\
\{P_{(v_2,v_3)}, P_{(v_3,v_4)},Q_{(v_3,v_4)}, P_1\} &\text{ if }\sigma=\sigma_6,\\
\{P_{(v_2,v_3)}, Q_{(v_2,v_4)}, P_{(v_3,v_4)}\} & \text{ if }\sigma=\sigma_7. 
\end{cases} \]
In addition, by the structural symmetry of $P$ and $Q$, we may assume the following:
\begin{itemize}
\item[-] If $i\in \{3,7\}$, then $P_0P_{(v_2,v_3)}\in E(\mathcal{H})$.
\item[-] If $i\in\{4,5\}$, then $P_0A \in E(\mathcal{H})$  for some  $A\in \{ P_{(v_2,v_3)}, P_{(v_3,v_4)}\}$.
\item[-] If $i=6$, then $P_0A \in E(\mathcal{H})$  for some  $A\in \{ P_{(v_2,v_3)}, P_{(v_3,v_4)},P_1\}$.
\end{itemize}
To be precise, when $\sigma=\sigma_3$, by the structural symmetry of $P$ and $Q$, the four paths $P_{(v_2,v_3)}$, $P_{(v_3,v_4)}$, $Q_{(v_2,v_3)}$, $Q_{(v_3,v_4)}$ play the same role in a configuration in Figure~\ref{fig:sigma}, so we may assume that $P_0$ and $P_{(v_2,v_3)}$ are adjacent in $\mathcal{H}$.
When $\sigma=\sigma_7$, we can replace $P$ and $Q$ by other two longest paths $P'$ and $Q'$ with  $V(P')\cap V(Q')=S$ and $E(P)\cup E(Q)=E(P')\cup E(Q')$, 
where $P' = P_{[p_0,v_1]} + Q_{[v_1,v_4]}+ P_{[v_4,v_2]} + Q_{[v_2,q_0]}$ and $Q' = P_{[p_1,v_4]}+ Q_{[v_4,v_2]} + P_{[v_2,v_1]}+Q_{[v_1,q_1]}$, 
Therefore we may assume that $P_0$ is adjacent to $P_{(v_2,v_3)}$ in $\mathcal{H}$. 
The other cases are similarly obtained.

From Claim~\ref{claim:p0}, we immediately reach a contradiction when $\sigma=\sigma_i$ for some $i\in\{3,5,7\}$. 
Suppose that $\sigma = \sigma_4$.
By Claim~\ref{claim:p0}, $P_0 P_{(v_2,v_3)} \not\in E(\mathcal{H})$, so a path in $\mathcal{H}$ from $P_0$ to $Q_0$ should start with $P_0$, $P_{(v_3,v_4)}$. 
By Claim~\ref{claim:XY} (i), there is no possible neighbor of $P_{(v_3,v_4)}$ in $\mathcal{H}$ from the replaceable pairs of $\sigma_4$. 
Suppose that $\sigma=\sigma_6$. 
By Claim~\ref{claim:p0}, $P_{(v_2,v_3)}$ and $P_{(v_3,v_4)}$ are not adjacent to $P_0$,
 so we have $N_{\mathcal{H}}(P_0)=\{P_1\}$. This is impossible since, by symmetry in $\sigma_6$, $N_{\mathcal{H}}(P_1)=\{P_0\}$, and this is a contradiction to the fact that $\mathcal{H}$ is connected.

Now suppose that $\sigma=\sigma_i$ for some $i\in \{1,2\}$.   
Note that each of three pairs $\{P_0,Q_0\}$, $\{P_1,Q_1\}$, $\{P_{(v_2,v_3)},Q_{(v_2,v_3)}\}$ is replaceable. 
We take a shortest path $\mathcal{R}$ in $\mathcal{H}$ from $\{P_0,Q_0\}$ to $\{P_1,Q_1\}$.
By symmetry, we may assume that the ends of $\mathcal{R}$  are $P_0$ and $P_1$.
By Claim~\ref{claim:XY}, 
$\mathcal{R}$ has length at least two,
and $\mathcal{H}$ is a subgraph of a graph in Figure~\ref{fig:H}.
\begin{figure}[h!]
  \centering
  \includegraphics[width=12cm, page = 6]{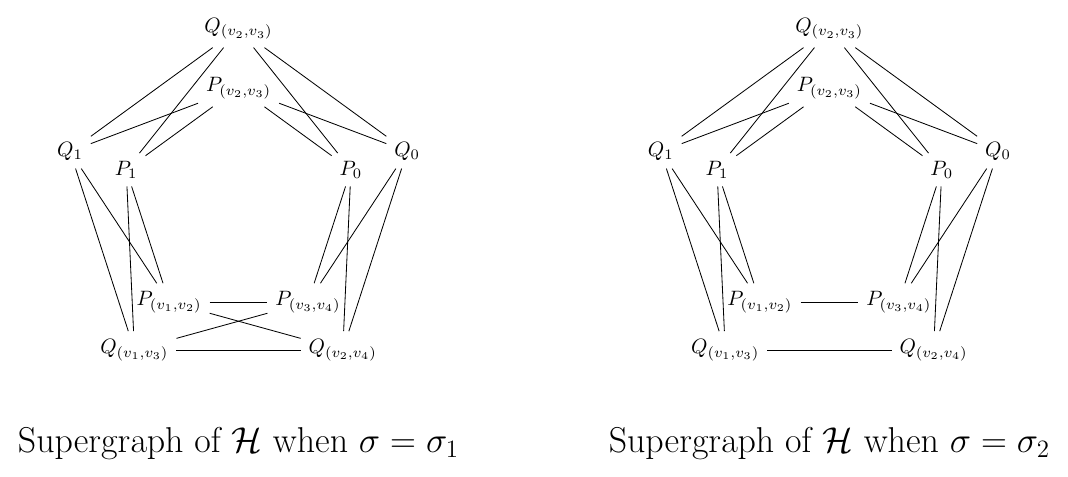}
  \caption{Illustrations for graphs containing $\mathcal{H}$.}
  \label{fig:H}
\end{figure}
By the structural symmetry of $P$ and $Q$,  the 
internal vertices of $\mathcal{R}$ are contained in $\mathcal{P} \setminus \{P_0,P_1\}$.
From $\mathcal{R}$, we can find a $(P_0,P_1)$-path $R$ in $G$ whose internal vertices are not in $V(Q) \cup V(P_0) \cup V(P_1)$, where $u \in V(P_0)$ and $w \in V(P_1)$ are ends of $R$.
Consider a longest path $P^*=P_{[p_0,v_1]}+Q_{[v_1,v_4]}+P_{[v_4,p_1]}$, and note that $R$ is a $(P^*_0,P^*_1)$-path in $G-S$ whose internal vertices are in $V(G) \setminus ( V(P^*) \cup V(Q))$.
Then $P'$ and $Q'$ are paths such that $\ell(P') + \ell(Q') = \ell(P) + \ell(Q) + 2\ell(R)$, where
\[P' = P^*_{[p_0,u]} + R + P^*_{[w,v_1]}+Q_{[v_1,q_0]}, \qquad  Q' = P^*_{[p_1,w]} + R + P^*_{[u,v_4]}+Q_{[v_4,q_1]}.\]
Thus, either $P'$ or $Q'$ is a path longer then a longest path, which is a contradiction.
\end{proof}

\section{Concluding remarks}\label{sec:remarks}

We extended the results of~\cite{gutierrez2020intersection} to provide further evidence for Conjecture~\ref{conj}. 
It is worth mentioning that natural generalizations for Claims~\ref{claim:XY} and~\ref{claim:p0} can be easily verified, so developing similar additional structural claims should be useful for proving Conjecture~\ref{conj} for $k\geq 6$.

In \cite{chen1998intersections}, it was shown that any two longest cycles in a $k$-connected graph share at least $ck^{3/5}$ vertices, where $c = 1/(\sqrt[3]{256}+3)^{3/5} \approx 0.2615$. By following the argument in \cite{chen1998intersections}, one can obtain the same conclusion for two longest paths in a $k$-connected graph; that is, every two longest paths in a $k$-connected graph have at least common $ck^{3/5}$  vertices.
As far as we know, this is the best known result so far, so any improvement on the growth rate would be interesting. 

\section*{Acknowledgements}
Eun-Kyung Cho was supported by Basic Science Research Program through the National Research Foundation of Korea(NRF) funded by the Ministry of Education(2020R1I1A1A0105858711).
Ilkyoo Choi was supported by the Basic Science Research Program through the National Research Foundation of Korea (NRF) funded by the Ministry of Education (NRF-2018R1D1A1B07043049), and also by the Hankuk University of Foreign Studies Research Fund.
Boram Park was supported by Basic Science Research Program through the National Research Foundation of Korea (NRF) funded by the Ministry of Science, ICT and Future Planning (NRF-2018R1C1B6003577).

\bibliographystyle{plain}

\begin{thebibliography}{10}

\bibitem{MR0232693}
{\em Theory of graphs}, volume 1966 of {\em Proceedings of the Colloquium held
  at Tihany, Hungary, September}.
\newblock Academic Press, New York-London; Akad\'{e}miai Kiad\'{o}, Budapest,
  1968.

\bibitem{bondy2008graph}
J.~A. Bondy and U.~S.~R. Murty.
\newblock {\em Graph theory}, volume 244 of {\em Graduate Texts in
  Mathematics}.
\newblock Springer, New York, 2008.

\bibitem{2017ChEhFeHeShYaYa}
Guantao Chen, Julia Ehrenm\"{u}ller, Cristina~G. Fernandes, Carl~Georg Heise,
  Songling Shan, Ping Yang, and Amy~N. Yates.
\newblock Nonempty intersection of longest paths in series-parallel graphs.
\newblock {\em Discrete Math.}, 340(3):287--304, 2017.

\bibitem{chen1998intersections}
Guantao Chen, Ralph~J. Faudree, and Ronald~J. Gould.
\newblock Intersections of longest cycles in {$k$}-connected graphs.
\newblock {\em J. Combin. Theory Ser. B}, 72(1):143--149, 1998.

\bibitem{gutierrez2020intersection}
Juan Gutiérrez.
\newblock On the intersection of two longest paths in $k$-connected graphs,
  2020.

\bibitem{hippchen2008intersections}
Thomas Hippchen.
\newblock Intersections of longest paths and cycles.
\newblock {\em Thesis, Georgia State University}, 2008.

\bibitem{2013ShZaZa}
Ayesha Shabbir, Carol~T. Zamfirescu, and Tudor~I. Zamfirescu.
\newblock Intersecting longest paths and longest cycles: a survey.
\newblock {\em Electron. J. Graph Theory Appl. (EJGTA)}, 1(1):56--76, 2013.

\bibitem{walther1974kreise}
Hansjoachim Walther and Heinz~J{\"u}rgen Voss.
\newblock {\em {\"u}ber Kreise in Graphen}.
\newblock Dt. Verlag d. Wiss., 1974.

\bibitem{1976Zamfirescu}
Tudor Zamfirescu.
\newblock On longest paths and circuits in graphs.
\newblock {\em Math. Scand.}, 38(2):211--239, 1976.

\bibitem{2001Zamfirescu}
Tudor Zamfirescu.
\newblock Intersecting longest paths or cycles: a short survey.
\newblock {\em An. Univ. Craiova Ser. Mat. Inform.}, 28:1--9, 2001.

\end{thebibliography}

\end{document}